\documentclass[a4paper,10pt]{article}

\usepackage[ascii]{inputenc}
\usepackage{amsmath,amsfonts,amssymb,amsthm}
\usepackage{mathrsfs}
\usepackage{xcolor}

\newtheoremstyle{theorem}{5pt}{5pt}{\itshape}{}{\bfseries}{.}{.5em}{}

\theoremstyle{theorem}
\newtheorem{theorem}{Theorem}
\newtheorem{lemma}[theorem]{Lemma}
\newtheorem{corollary}[theorem]{Corollary}
\newtheorem{proposition}[theorem]{Proposition}
\newtheorem{remark}[theorem]{Remark}

\usepackage{titlesec}
\titleformat{\section}
{\normalfont\bfseries}{\large\thesection}{1em}{\large}
\titlespacing*{\section}{0pt}{3.5ex plus 1ex minus .2ex}{2.3ex plus .2ex}
\titleformat{\subsection}
{\normalfont\bfseries}{\thesubsection}{1em}{\normalsize}
\titlespacing*{\subsection}{0pt}{3.5ex plus 1ex minus .2ex}{2.3ex plus .2ex}

\makeatletter
\def\Ddots{\mathinner{\mkern1mu\raise\p@
\vbox{\kern7\p@\hbox{.}}\mkern2mu
\raise4\p@\hbox{.}\mkern2mu\raise7\p@\hbox{.}\mkern1mu}}
\makeatother

\allowdisplaybreaks[3]

\begin{document}

\title{On Short Sums Involving Fourier Coefficients of Maass Forms}
\author{Jesse J\"a\"asaari\footnote{Department of Mathematics and Statistics, University of Helsinki}{ }}
\date{}
\maketitle

\begin{abstract}
\noindent 

\noindent We study sums of Hecke eigenvalues of Hecke-Maass cusp forms for the group $\mathrm{SL}(n,\mathbb Z)$, with general $n\geq 3$, over certain short intervals under the assumption of the generalised Lindel\"of hypothesis and a slightly stronger upper bound concerning the exponent towards the Ramanujan-Petersson conjecture that is currently known. In particular, in this case we evaluate the second moment of the sums in question asymptotically.
\end{abstract}

\section{Introduction}

Let $f$ be a Maass cusp form of type $\nu\in\mathbb C^{n-1}$ for the full modular group $\mathrm{SL}(n,\mathbb Z)$ with Fourier coefficients $A(m_1,...,m_{n-1})=A_f(m_1,...,m_{n-1})$. Throughout the article, let $n\geq 3$ be a positive integer. The Fourier-Whittaker expansion of $f$ is given by
\begin{align*}
&f(z)=\sum_{\gamma\in U(n-1,\mathbb Z)\backslash\mathrm{SL}(n-1,\mathbb Z)}\sum_{m_1=1}^\infty\cdots\sum_{m_{n-2}=1}^\infty\sum_{m_{n-1}\neq 0}\frac{A(m_1,...,m_{n-1})}{\prod_{k=1}^{n-1}|m_k|^{k(n-k)/2}}\\
&\qquad\qquad\cdot W_{\text{Jacquet}}\left(M\begin{pmatrix}
\gamma & \\
 & 1
\end{pmatrix}z,\nu,\psi_{\left(1,...,1,\frac{m_{n-1}}{|m_{n-1}|}\right)}\right),
\end{align*}
where
\begin{align*}
M=\begin{pmatrix}
m_1\cdots m_{n-2}|m_{n-1}| & & \\
 &\ddots &\\
 & &  m_1m_2 & \\
 & & & m_1 & \\
 & & & & 1
\end{pmatrix},
\end{align*}
$U(n-1,\mathbb Z)$ is the group of $(n-1)\times(n-1)$ upper triangular matrices with ones on the diagonal and integral entries above the diagonal, $\psi_{\left(1,...,1,m_{n-1}/|m_{n-1}|\right)}$ is a certain character, and $W_{\text{Jacquet}}$ is the Jacquet-Whittaker function of type $\nu$ for the character $\psi_{\left(1,...,1,m_{n-1}/|m_{n-1}|\right)}$. For more details we refer to Goldfeld's book \cite{Goldfeld}. We further assume that the Maass cusp form $f$ is an eigenfunction for the full Hecke ring and normalised so that $A(1,...,1)=1$.

In this case it is known that the eigenvalue of $f$ under the $m^{\mathrm{th}}$ Hecke operator $T_m$ (see Section 4) is given by $A(m,1,...,1)$. These coefficients $A(m,1,...,1)$ also appear as the Dirichlet series coefficients in the standard Godement-Jacquet $L$-function attached to such Maass cusp form and thus it is natural to concentrate on them. Such coefficients $A(m,1,...,1)$ are fascinating number theoretic objects and they have been studied extensively as are the Fourier coefficients of holomorphic cusp forms and Maass cusp forms in the classical situation $n=2$. 

Obtaining estimates for the sum of Hecke eigenvalues of cusp forms is a classical problem with a long history as we will now explain. For the Fourier coefficients of a holomorphic cusp form, denoted by $a(m)$, the trivial bound for the long sum
\begin{align*}
\sum_{m\leq x}a(m)
\end{align*}
is $\ll_\varepsilon x^{1+\varepsilon}$, for every $\varepsilon>0$. First to improve this was Hecke \cite{Hecke} who showed essentially squareroot cancellation and this was soon after sharpened by Walfisz \cite{Walfisz}. Then Rankin \cite{Rankin} showed that one has an estimate of the form
\begin{align*}
\sum_{m\leq x}a(m)\ll x^{2/5},
\end{align*}
which was the sharpest result for a long time. The currently best known upper bound is $\ll_\varepsilon x^{1/3}(\log x)^{-\delta+\varepsilon}$ for $\delta=(8-3\sqrt 6)/10$ proved by Rankin himself \cite{Rankin2}. For the classical Maass cusp form coefficients $t(m)$, it is known that 
\begin{align*}
\sum_{m\leq x}t(m)\ll_\varepsilon x^{1/3+\vartheta/3+\varepsilon},
\end{align*}
where $\vartheta\geq 0$ is the exponent towards the Ramanujan-Petersson conjecture for classical Maass cusp forms \cite{Hafner-Ivic}. It is known that $\vartheta\leq 7/64$. Currently the best known unconditional result for classical Maass cusp form coefficients is $\ll_\varepsilon x^{1027/2827+\varepsilon}$, which is due to L\"u \cite{Lu2}. It is a folklore conjecture that the correct upper bound is $\ll_\varepsilon x^{1/4+\varepsilon}$, for both holomorphic cusp forms and classical Maass cusp forms. 

Concerning the higher rank analogue, Goldfeld and Sengupta \cite{Goldfeld-Sengupta} have recently shown that for the Fourier coefficients of a $\mathrm{GL}(n)$ Maass cusp form, the upper bound
\begin{align*}
\sum_{m\leq x}A(m,1,...,1)\ll_\varepsilon x^{(n^3-1)/(n^3+n^2+n+1)+\varepsilon}
\end{align*}  
holds for any $n\geq 3$. Again the trivial bound for the sum is $\ll_\varepsilon x^{1+\varepsilon}$. This was recently slightly improved further by Meher and Murty \cite{Meher-Murty}. The conjectural upper bound in this case is $\ll_\varepsilon x^{1/2-1/2n+\varepsilon}$. 

It is natural to study analogous problems for shorter summation ranges $[x,x+\Delta]$ with $\Delta=o(x)$. Intuitively, studying short sums makes sense as one might suspect that shorter intervals capture the erratic nature of the Fourier coefficients better than longer intervals. Furthermore, when $\Delta$ is small compared to $x$, studying short sums is analogous to studying classical error terms in analytic number theory, such as the error term in Dirichlet's divisor problem, in the short intervals.
 
Pointwise bounds for short sums involving Fourier coefficients of cusp forms with an exponential twist (of which the plain sum of coefficients corresponds to the case of a trivial twist) have been obtained first by Jutila \cite{Jutila1987a} and later by Ernvall-Hyt\"onen and Karppinen \cite{Ernvall-Hytonen--Karppinen} in the $\mathrm{GL}(2)$-setting for holomorphic cusp forms. Recently analogues of many results of \cite{Ernvall-Hytonen--Karppinen} have been proved for sums involving Fourier coefficients of classical Maass cusp forms \cite{Jaasaari-Vesalainen1}.

In the present article we evaluate the mean square of sums of Hecke eigenvalues asymptotically over certain short intervals in the general $\mathrm{GL}(n)$-situation assuming the generalised Lindel\"of hypothesis for the $L$-function attached to the underlying cusp form in the $t$-aspect and a weak version of the Ramanujan-Petersson conjecture. Previously an analogous result has been established for the error term in Dirichlet's divisor problem for the $k$-fold divisor function $d_k$, given by
\begin{align*}
\Delta_k(x):=\sum_{m\leq x}d_k(m)-\text{Res}_{s=1}\left(\zeta^k(s)\frac{x^s}s\right),
\end{align*} 
by Lester \cite{Lester} under the Lindel\"of hypothesis for the $\zeta$-function and we follow his strategy. Many of the details are similar, but we present the whole argument for the sake of completeness as only a bound of the form $A(m,1,...,1)\ll_\varepsilon m^{\vartheta+\varepsilon}$, for some fixed $\vartheta\geq 0$, is known for the Hecke eigenvalues. It is important to keep track of $\vartheta$ because we only know that $\vartheta\leq 1/2-1/(n^2+1)$. Indeed, our main theorems are conditional on the assumption $\vartheta<1/2-1/n$. 
 
While analytic number theory of automorphic forms has seen many advances in the classical $\mathrm{GL}(2)$-setting, the results are more sporadic in the case $n\geq 3$. There are not many statements which are currently known to hold for an individual (contrast to on average over a family of) cusp form on $\mathrm{GL}(n)$ for arbitrary $n$. The best known results of this type are the approximations to the Ramanujan-Petersson conjecture discussed below. The main results in the present article add further examples of such properties assuming generic hypothesis which are expected to be true.

This article is organised as follows. In Section 2 we introduce the statements of the main theorems. In Section 4 we collect some facts and results needed in the proofs. The penultimate section contains the proof of Theorem 1 and Theorem 2 is proved in Section 6.  

\section{The main results}


The average behaviour of short rationally additively twisted exponential sums weighted by Fourier coefficients of holomorphic cusp forms has been studied e.g. by Jutila \cite{Jutila--short}, Ernvall-Hyt\"onen \cite{Ernvall-Hytonen, Ernvall-Hytonen2}, and Vesalainen \cite{Vesalainen}. In the higher rank case, the mean square of long rationally additively twisted sums involving Fourier coefficients of $\mathrm{SL}(3,\mathbb Z)$ Maass cusp forms has been considered in \cite{Jaasaari-Vesalainen}.

Here we study sums of Fourier coefficients of Hecke-Maass cusp forms for $\mathrm{SL}(n,\mathbb Z)$ over short intervals under certain generic assumptions. However, the method of the proof is slightly different compared to the works mentioned above. Jutila's method is not applicable here essentially for two reasons; first one being that trigonometric polynomials in the truncated $\mathrm{GL}(n)$-Voronoi summation formula are more complicated than in the lower rank setting and the other one is that the error term in the relevant truncated Voronoi summation formula gives larger contribution than the expected main term. Instead, we follow Lester \cite{Lester} who treats a similar problem for the error term of the Dirichlet divisor problem for the $k$-fold divisor function by combining Jutila's method with the one of Selberg \cite{Selberg2}. Selberg's method can also be applied to other problems concerning automorphic forms, see e.g. \cite{Milinovich-Turnage-Butterbaugh}. The assumptions concerning the truth of the generalised Lindel\"of hypothesis and a weak form of the Ramanujan-Petersson conjecture are needed to guarantee that the expected error term in the truncated Voronoi summation formula is small enough on average. 

In both of the main results, let the underlying Hecke-Maass cusp form be $f$ with Hecke eigenvalues $A(m,1,...,1)$. Our first main result computes the variance of short sums of these coefficients. These types of averages appear for example when studying the value distribution of said short sums over intervals of certain length. 

\begin{theorem}\label{plainsum}
Let $f$ be a Hecke-Maass cusp form for $\mathrm{SL}(n,\mathbb Z)$ normalised so that $A(1,...,1)=1$. Assume the generalised Lindel\"of hypothesis for $L(s,f)$ in the $t$-aspect and that the exponent towards the Ramanujan-Petersson conjecture satisfies $0\leq\vartheta<1/2-1/n$. Furthermore, suppose that $2\leq L\ll_\varepsilon X^{1/(n(n-1+2n\vartheta))-\varepsilon}$ for some fixed $\varepsilon>0$ and that $L=L(X)\longrightarrow\infty$ as $X\longrightarrow\infty$. Then we have
\begin{align*}
\frac1X\int_X^{2X}\left|\sum_{x\leq m\leq x+x^{1-1/n}/L}A(m,1,...,1)\right|^2\mathrm d x\sim C_f\cdot\frac{X^{1-1/n}}L.
\end{align*}
\end{theorem}

\noindent Here $C_f$ is a constant given by
\begin{align}\label{constant}
C_f:=\frac{2^{1-1/n}-1}{2n-1}\cdot r_f\cdot H_f(1),
\end{align}
where $r_f$ is the residue of the Rankin-Selberg $L$-function $L(s,f\times\widetilde f)$ attached to the underlying Maass cusp form $f$ at $s=1$, and $\widetilde f$ is the dual Maass form of the form $f$. It is given by 
\begin{align*}
r_f=\frac{4\pi^{n^2/2}}{n\cdot w(f)}\|f\|^2.
\end{align*}
For the proof of this, see Appendix A in \cite{Kowalski-Ricotta}. 
The Petersson norm of $f$ is given by
\begin{align*}
\|f\|^2:=\int_{\mathrm{SL}(n,\mathbb Z)\backslash\mathbb H^n}|f(z)|^2\,\mathrm d^* z,
\end{align*} 
where $\mathrm d^*z$ is the $\mathrm{GL}(n,\mathbb R)$-invariant measure on the generalised upper-half plane $\mathbb H^n\simeq\mathrm{SL}(n,\mathbb R)/\mathrm{SO}(n,\mathbb R)$, see Section 1.5 of \cite{Goldfeld}. Here the group $\mathrm{GL}(n,\mathbb R)$ acts on $\mathbb H^n$ by left matrix multiplication. Furthermore, 
\begin{align*}
w(f):=\prod_{1\leq j\leq n}\Gamma\left(\frac{1+2\Re(\lambda_j(\nu))}2\right)\prod_{1\leq j<k\leq n}\left|\Gamma\left(\frac{1+\lambda_j(\nu)+\overline{\lambda_k(\nu)}}2\right)\right|^2,
\end{align*}
where $\lambda_j(\nu)$, $j=1,...,n$, are the Langlands parameters of the form $f$. These are complex numbers expressed in terms of the type $\nu=(\nu_1,...,\nu_{n-1})\in\mathbb C^{n-1}$ of $f$. Finally, the constant $H_f(1)$ is given by
\begin{align*}
H_f(1):=\prod_p P_n(\alpha_p(f),\alpha_p(\widetilde f),p^{-1}),
\end{align*}
where $P_n$ is the polynomial defined in (\ref{schuridentity}) below, $\alpha_p(f):=\{\alpha_{1,p}(f),...,\alpha_{n,p}(f)\}$ is the set of Satake parameters of $f$ at prime $p$ and it turns out that $\alpha_p(\widetilde f)=\overline{\alpha_p(f)}:=\{\overline{\alpha_{1,p}(f)},...,\overline{\alpha_{n,p}(f)}\}$. The fact that $H_f(1)$ is non-zero is shown in \cite[Appendix B]{Kowalski-Ricotta}.

The other main theorem computes the mean square of the sum of Hecke eigenvalues over certain short intervals of fixed length.
\begin{theorem}\label{theorem2}
Let $f$ be a Hecke-Maass cusp form for $\mathrm{SL}(n,\mathbb Z)$ normalised so that $A(1,...,1)=1$. Suppose that $X^{1-1/n+\varepsilon}\ll_\varepsilon\Delta\ll_\varepsilon X^{1-\varepsilon}$ for some small fixed $\varepsilon>0$ and that the generalised Lindel\"of hypothesis for $L(s,f)$ holds in the $t$-aspect. Suppose also that the exponent towards the Ramanujan-Petersson conjecture satisfies $\vartheta<1/2-1/n$. Then we have
\begin{align*}
\frac1X\int_X^{2X}\left|\sum_{x\leq m\leq x+\Delta}A(m,1,...,1)\right|^2\,\mathrm d x\sim B_f\cdot X^{1-1/n},
\end{align*}
where 
\begin{align*}
B_f:=\frac1{\pi^2}\cdot\frac{2^{2-1/n}-1}{2n-1}\sum_{m=1}^\infty\frac{|A(m,1,...,1)|^2}{m^{1+1/n}}.
\end{align*}
\end{theorem}
\noindent The fact that $B_f$ is finite follows from (\ref{gl(n)rankin-selberg}) below and partial summation. This partly generalises results of Ivi\'c \cite{Ivic}, Jutila \cite{Jutila--short}, and Vesalainen \cite{Vesalainen} to the higher rank setting, and is an analogue to Lester's result \cite{Lester} in the setting of cusp forms. 

\begin{remark}
Theorem \ref{theorem2} is not expected to hold in the range $\Delta\ll_{\varepsilon'} X^{1-1/n-\varepsilon'}$ as in that range the sum of coefficients over the interval $[x,x+\Delta]$, with $x\asymp X$, is conjectured to be bounded from above by $\sqrt\Delta$. 
\end{remark}

\section{Notation}

The symbols $\ll$, $\gg$, $\asymp$, $O$, and $\sim$ are used for the usual asymptotic notation: for complex valued functions $f$ and $g$ in some set $X$, the notation $f\ll g$ means that $\left|f(x)\right|\leqslant C\left|g(x)\right|$ for all $x\in X$ for some implicit constant $C\in\mathbb R_+$. When the implied constant depends on some parameters $\alpha,\beta,\ldots$, we use $\ll_{\alpha,\beta,\ldots}$ instead of mere $\ll$. The notation $g\gg f$ means $f\ll g$, and $f\asymp g$ means $f\ll g\ll f$.

All the implicit constants are allowed to depend on the underlying Maass cusp form and on $\varepsilon$, which denotes an arbitrarily small fixed positive number, which may not be the same on each occurrence, unless stated otherwise.

As usual, we write $e(x)$ for $e^{2\pi ix}$. The notation $\prod_p$ means the product over primes. The real and imaginary parts of a complex number $s$ are denoted by $\Re(s)$ and $\Im(s)$, respectively. Sometimes we also write $s=\sigma+it$ with $\sigma,t\in\mathbb R$. Finally, $\langle t\rangle$ stands for $(1+|t|^2)^{1/2}$.

\section{Useful results}

We start by recalling a few facts about higher rank Hecke operators and automorphic $L$-functions. By analogue to the classical situation, it follows that for every integer $m\geq 1$, we have a Hecke operator given by
\begin{align*}
T_mf(z):=\frac1{m^{n-1/2}}\sum_{\substack{\prod_{\ell=1}^n c_\ell=m\\
0\leq c_{i,\ell}<c_\ell\,\,(1\leq i<\ell\leq n)}}f\left(\begin{pmatrix}
c_1 & c_{1,2} &\cdots & c_{1,n}\\
 & c_2 & \cdots & c_{2,n}\\
 & & \ddots & \vdots\\
 & & & c_n
\end{pmatrix}\cdot z\right)
\end{align*}
acting on the space $L^2(\mathrm{SL}(n,\mathbb Z)\backslash\mathbb H^n)$ of square-integrable automorphic functions (which contains the space of Maass cusp forms). Unlike in the classical situation, these operators are not self-adjoint but they are normal. If a Maass cusp form $f$ is an eigenfunction of every Hecke operator, it is called a Hecke-Maass cusp form. We remark that if the Fourier coefficient $A(1,...,1)$ of a Hecke-Maass cusp form is zero, then the form vanishes identically. For more theory of Hecke operators for $\mathrm{SL}(n,\mathbb Z)$, see \cite[Section 9.3.]{Goldfeld}.

The Fourier coefficients of Hecke-Maass cusp forms are related to the Satake parameters via the work of Shintani \cite{Shintani} together with results of Casselman and Shalika \cite{Casselman-Shalika}. They showed that for any prime number $p$ and $\beta_1,...,\beta_n\in\mathbb Z_{\geq 0}$ one has
\begin{align*}
A_f(p^{\beta_1},...,p^{\beta_{n-1}})=S_{\beta_{n-1},...,\beta_1}(\alpha_{1,p}(f),...,\alpha_{n,p}(f)), 
\end{align*} 
where 
\begin{align*}
& S_{\beta_{n-1},...,\beta_1}(x_1,...,x_n)\\
&:=\frac1{V(x_1,...,x_n)}\det\left[\begin{pmatrix}
x_1^{n-1+\beta_{n-1}+\cdots+\beta_1} & \cdots & x_n^{n-1+\beta_{n-1}+\cdots+\beta_1}\\
\vdots & \vdots & \vdots\\
x_1^{2+\beta_{n-1}+\beta_{n-2}} & \cdots & x_n^{2+\beta_{n-1}+\beta_{n-2}}\\
x_1^{1+\beta_{n-1}} & \cdots & x_n^{1+\beta_{n-1}}\\
1 & \cdots & 1
\end{pmatrix}
\right]
\end{align*}
is a Schur polynomial, and $V(x_1,...,x_n)$ is the Vandermonde determinant given by
\begin{align*}
V(x_1,...,x_n):=\prod_{1\leq i<j\leq n}(x_i-x_j).
\end{align*}
Kowalski and Ricotta proved in \cite[Proposition B.1]{Kowalski-Ricotta} that there exists a polynomial $P_n(\textbf{x},\textbf{y},T)$, where $\textbf{x}=(x_1,...,x_n)$, $\textbf{y}=(y_1,...,y_n)$, and $T$ are indeterminates, such that
\begin{align}\label{schuridentity}
\sum_{k\geq 0}S_{0,...,0,k}(\textbf{x})S_{0,...,0,k}(\textbf{y})T^k=\frac{P_n(\textbf{x},\textbf{y},T)}{\prod_{1\leq j,k\leq n}(1-x_jy_kT)}.
\end{align}
Next, we define an important notion of a dual Maass cusp form. Let $f$ be a Maass cusp form of type $(\nu_1,...,\nu_{n-1})\in\mathbb C^{n-1}$ for $\mathrm{SL}(n,\mathbb Z)$. Then 
\begin{align*}
\widetilde f(z):=f(w\cdot\,^t(z^{-1})\cdot w),\quad\text{where} \qquad w=\begin{pmatrix}
& & & (-1)^{\lfloor n/2\rfloor}\\
& & 1 & \\
& \Ddots & &\\
 1 & & &
\end{pmatrix},
\end{align*}
is a Maass cusp form of type $(\nu_{n-1},...,\nu_1)\in\mathbb C^{n-1}$ for $\mathrm{SL}(n,\mathbb Z)$. We say that $\widetilde f$ is a dual Maass cusp form of $f$. It turns out that 
\begin{align}\label{dualmaasscoeff}
A_f(m_1,...,m_{n-1})=A_{\widetilde f}(m_{n-1},...,m_1)
\end{align}
for every $m_1,...,m_{n-1}\geq 1$. 

The Fourier coefficients of Hecke-Maass cusp form satisfy the multiplicativity relations
\begin{align*}
A(m,1,...,1)A(m_1,...,m_{n-1})=\sum_{\substack{\prod_{\ell=1}^n c_\ell=m\\
c_j|m_j\,\,\text{for }1\leq j\leq n-1}}A\left(\frac{m_1c_n}{c_1},\frac{m_2c_1}{c_2},...,\frac{m_{n-1}c_{n-2}}{c_{n-1}}\right)
\end{align*}
for positive integers $m,m_1,...,m_{n-2}$, and a non-negative integer $m_{n-1}$. Furthermore, the relation
\begin{align}\label{multiplicativity}
A(m_1,...,m_{n-1})A(m_1',...,m_{n-1}')=A(m_1m_1',...,m_{n-1}m_{n-1}')
\end{align}
holds if $(m_1\cdots m_{n-1},m_1'\cdots m_{n-1}')=1$. For the proofs of these facts, see \cite[Theorem 9.3.11.]{Goldfeld}

For a Hecke eigenfunction, one can use M\"obius inversion to show that the relation
\begin{align*}
A(m_1,...,m_{n-1})=\overline{A(m_{n-1},...,m_1)}
\end{align*} 
holds \cite[Theorem 9.3.6, Theorem 9.3.11, Addendum]{Goldfeld}. In particular, together with the relation (\ref{dualmaasscoeff}) this yields that 
\begin{align*}
\overline{A_f(m,1,...,1)}=A_{\widetilde f}(m,1,...,1).
\end{align*}
Also, it follows that $|A(m,1,..,1)|=|A(1,...,1,m)|$. Given a Hecke-Maass cusp form $f$ of type $\nu\in\mathbb C^{n-1}$ for $\mathrm{SL}(n,\mathbb Z)$ with Hecke eigenvalues $A(m,1,...,1)$, the associated $L$-series is given by
\begin{align*}
L(s,f):=\sum_{m=1}^\infty\frac{A(m,1,...,1)}{m^s},
\end{align*}
which converges for large enough $\Re(s)$. This has an entire continuation to the whole complex plane via the functional equation
\begin{align}\label{functionalequation}
L(s,f)=\pi^{ns-n/2}\frac{G(1-s,\widetilde f)}{G(s,f)}L(1-s,\widetilde f),
\end{align}
where
\begin{align*}
G(s,f):=\prod_{j=1}^n\Gamma\left(\frac{s-\lambda_j(\nu)}2\right)\quad\text{and so}\quad G(s,\widetilde f)=\prod_{j=1}^n\Gamma\left(\frac{s-\widetilde\lambda_j(\nu)}2\right).
\end{align*} 
Recall that here $\lambda_j(\nu)$ and $\widetilde{\lambda}_j(\nu)$ are the Langlands parameters of $f$ and $\widetilde f$, respectively. This produces an $L$-function attached to the form $f$ called the Godement-Jacquet $L$-function.

An elementary application of Stirling's formula says, that when $s$ lies in the vertical strips $-\delta\leq\Re(s)\leq 1+\delta$, for a small fixed $\delta>0$, and has sufficiently large imaginary part, the multiple $\Gamma$-factors can be replaced by a single quotient of two $\Gamma$-factors \cite{Ernvall-Hytonen--Jaasaari--Vesalainen}:
\begin{align*}
\frac{G(1-s,\widetilde f)}{G(s,f)}=n^{ns-n/2}\frac{\Gamma\left(\frac{1-ns}2\right)}{\Gamma\left(\frac{ns-(n-1)}2\right)}\left(1+O(|s|^{-1})\right).
\end{align*}
The main theorems of the present paper are conditional on the generalised Lindel\"of hypothesis in the $t$-aspect. It states that on the critical line $\sigma=1/2$ an estimate of the form $L(1/2+it,f)\ll_\varepsilon \langle t\rangle^\varepsilon$ holds for every $\varepsilon>0$. For more detailed discussion about this conjecture, see \cite{Iwaniec-Sarnak}.

The Rankin-Selberg $L$-function of two Hecke-Maass cusp forms $f$ and $g$ for $\mathrm{SL}(n,\mathbb Z)$ is given by
\begin{align*}
L(s,f\times g):=\zeta(ns)\sum_{m_1,...,m_{n-1}\geq 1}\frac{A_f(m_1,...,m_{n-1})\overline{A_g(m_1,...,m_{n-1})}}{(m_1^{n-1}m_2^{n-2}\cdots m_{n-1})^s},
\end{align*}
which converges for large enough $\Re (s)$. This $L$-series has an analytic continuation to the whole complex plane if $g\neq\widetilde f$ and a meromorphic continuation to $\mathbb C$ with a simple pole at $s=1$ if $g=\widetilde f$. If we set
\begin{align*}
\Lambda(s,f\times g):=\prod_{i=1}^n\prod_{j=1}^n\pi^{(-s+\lambda_i(\nu_f)+\overline{\lambda_j(\nu_g)})/2}\Gamma\left(\frac{s-\lambda_j(\nu_f)-\overline{\lambda_j(\nu_g)}}2\right)L(s,f\times g),
\end{align*}
then the functional equation
\begin{align}\label{rankin-selberg-fe}
\Lambda(s,f\times g)=\Lambda(1-s,\widetilde f\times\widetilde g)
\end{align}
holds; see \cite[Theorem 12.1.4]{Goldfeld}. 

If $L(s,f)$ has an Euler product representation
\begin{align*}
L(s,f)=\sum_{m=1}^\infty\frac{A(m,1,...,1)}{m^s}=\prod_p\prod_{j=1}^n(1-\alpha_{j,p}(f)p^{-s})^{-1}
\end{align*}
for large enough $\Re(s)$, and similar representation holds for $g$ with parameters $\alpha_{j,p}(g)$, then also the Rankin-Selberg $L$-function has an Euler product given by
\begin{align*}
L(s,f\times g)=\prod_p\prod_{k=1}^n\prod_{\ell=1}^n(1-\alpha_{k,p}(f)\overline{\alpha_{\ell,p}(g)}p^{-s})^{-1}.
\end{align*}
Recall that here the complex numbers $\alpha_{j,p}(f)$ are called the Satake parameters of the underlying Hecke-Maass cusp form $f$. 
Analytic properties of $L(s,f\times \widetilde f)$ imply that 
\begin{align}\label{gl(n)rankin-selberg}
\sum_{m_1^{n-1}m_2^{n-2}\cdots m_{n-1}\leq x}|A(m_1,m_2,...,m_{n-1})|^2\sim r_f\cdot x,
\end{align}
where $r_f$ is as before, see \cite[Proposition 12.1.6, Remark 12.1.8]{Goldfeld}. This result can be interpreted as saying that the Fourier coefficients $A(m_1,...,m_{n-1})$ are essentially of constant size on average. However, pointwise bounds for the Fourier coefficients are quite far from the expected truth. The Ramanujan-Petersson conjecture predicts that an estimate of the form $A(m,1,...,1)\ll_\varepsilon m^{\varepsilon}$ holds for every $\varepsilon>0$. There are however approximations towards this conjecture. Let $\vartheta=\vartheta(n)\geq 0$ be a real number so that the estimate $A(m,1,...,1)\ll_\varepsilon m^{\vartheta+\varepsilon}$ holds. It is easy to see that one can take $\vartheta=1/2$ \cite[Proposition 12.1.6]{Goldfeld}, but currently it is known that $\vartheta\leq 1/2-1/(n^2+1)$. This result is due to \cite{Luo-Rudnick-Sarnak}. For small values of $n$ sharper results are known. We have $\vartheta(2)\leq 7/64$, $\vartheta(3)\leq 5/14$ and $\vartheta(4)\leq 9/22$ \cite{Kim-Sarnak}. The Ramanujan-Petersson conjecture predicts that the value $\vartheta(n)=0$ is admissible for every $n\geq 2$. An equivalent estimate holds for the Satake parameters of the underlying form $f$. Namely, we have $\alpha_{j,p}(f)\ll_\varepsilon p^{\vartheta(n)+\varepsilon}$ for every prime $p$. 

It follows from (\ref{gl(n)rankin-selberg}), for $\delta\in\mathbb R_+$, that
\begin{align}\label{final-bound}
\sum_{m=1}^\infty\frac{\left|A(1,...,1,m)\right|^2}{m^{1+\delta}}\ll_\delta 1\qquad\text{and}\qquad\sum_{m=1}^\infty\frac{\left|A(1,...,1,m)\right|}{m^{1+\delta}}\ll_\delta 1.
\end{align}
In the course of the proof of Proposition \ref{erroronaverage} we will come across certain complex line integrals involving $\Gamma$-functions. More precisely, these integrals are of the form
\[\Omega_{\nu,k}(y;\delta,Y):=\frac1{2\pi i}\int\limits_{-\delta-iY}^{-\delta+iY}\frac{\Gamma\!\left(\frac{1-ns}2\right)}{\Gamma\!\left(\frac{ns+1}2+\nu-\frac n2\right)}\left(s+\Lambda\right)^{-k}y^s\,\mathrm ds,\]
where integration is along a straight line segment, and where $\nu$ and $k$ are non-negative integers, and $y$ and $Y$ are positive real numbers. The parameter $\Lambda$ is a large positive real number, which will depend on $n$ and the underlying Maass cusp form. The parameter $\delta$ will be a sufficiently small positive real constant. All the implicit constants in the following are going to depend on $n$, $\delta$ and $\Lambda$. It is proved in \cite{Jaasaari-Vesalainen} that the following lemma holds.

\begin{lemma}\label{obtaining-bessel-functions}(Lemma 8 in \cite{Jaasaari-Vesalainen})
Let $\nu$ and $k$ be non-negative integers, and let $y,Y\in\left[1,\infty\right[$ with $y<(nY/2)^n$. Then
\begin{multline*}
\Omega_{\nu,k}(y;\delta,Y)=\left(\frac n2\right)^{k-1}y^{1/2+(1-\nu-k)/n}\,J_{\nu+k-n/2}(2y^{1/n})+O(1)\\+O\left(Y^{n/2-\nu-k+n\delta}\right)+O\!\left(Y^{n/2-\nu-k}\frac1{\log\frac{n^nY^n}{2^ny}}\right).
\end{multline*}
\end{lemma}

\noindent Using the asymptotics of $J$-Bessel functions for $y\gg1$, we get the following corollary.
\begin{corollary}\label{key-corollary}
Let $y,Y\in\left[1,\infty\right[$ with $y<(nY/2)^n$. Then
\begin{multline*}
\Omega_{0,1}(y;\delta,Y)
=\frac1{\sqrt\pi}\,y^{1/2-1/(2n)}\,\cos\!\left(2y^{1/n}+\frac{(n-3)}4\,\pi\right)\\+O(y^{1/2-1/(2n)-1/n})+O(Y^{n/2-1+n\delta})+O\!\left(Y^{n/2-1}\frac1{\log\frac{n^nY^n}{2^ny}}\right).
\end{multline*}
\end{corollary}
\noindent We also need an asymptotic formula for the sum of the coefficients $|A(m,1,...,1)|^2$. The proof of the following theorem combines methods from \cite{Kowalski-Ricotta} and \cite{Meher-Murty}. 
\begin{theorem}\label{squaresum}
Let $f$ be a Hecke-Maass cusp form for the group $\mathrm{SL}(n,\mathbb Z)$ normalised so that $A(1,...,1)=1$. Then 
\begin{align*}
\sum_{m\leq x}|A(m,1,...,1)|^2\sim r_f\cdot H_f(1)\cdot x,
\end{align*}
where $r_f$ and $H_f(1)$ are as given above. 
\end{theorem}
\begin{proof}
We start by defining a Dirichlet series
\begin{align*}
D_f(s):=\sum_{m\geq 1}\frac{|A(m,1,...,1)|^2}{m^s},
\end{align*}
which is absolutely convergent for $\Re(s)>1$ due to (\ref{gl(n)rankin-selberg}) and defines a holomorphic function on this half-plane. Since $f$ is a Hecke eigenform, the coefficients $A(m,1,...,1)$ are multiplicative by using (\ref{multiplicativity}). Therefore we have
\begin{align*}
D_f(s)=\prod_p\sum_{k\geq 0}\frac{|A(p^k,1,...,1)|^2}{p^{ks}}=:\prod_p D_{f,p}(s).
\end{align*}
Furthermore, by applying (\ref{schuridentity}) with $\textbf{x}=\alpha_p(f)$, $\textbf{y}=\alpha_p(\widetilde f)$ and $T=p^{-s}$, we have
\begin{align*}
D_{f,p}(s)=\frac{P_n(\alpha_p(f),\alpha_p(\widetilde f),p^{-s})}{\prod_{1\leq j,k\leq n}(1-\alpha_{j,p}(f){\alpha_{k,p}(\widetilde f)}p^{-s})}
\end{align*}
for any prime $p$, where $P_n$ is the polynomial given by (\ref{schuridentity}). Hence, by using the explicit description of $P_n(\alpha_p(f),\alpha_p(\widetilde f),p^{-s})$ \cite[Proposition B.1 (3)]{Kowalski-Ricotta}, estimates $\alpha_p(f)$, $\alpha_p(\widetilde f)\ll_\varepsilon p^{\vartheta+\varepsilon}$ and estimating by absolute values, the quotient
\begin{align}\label{synttarit}
\frac{D_f(s)}{L(s,f\times\widetilde f)}=\prod_p P_n(\alpha_p(f),\alpha_p(\widetilde f),p^{-s})=:H_f(s)
\end{align}
defines a bounded holomorphic function on the half-plane $\Re(s)>1/2+\vartheta$. 

Hence, writing $H_f(s)$ as a Dirichlet series
\begin{align*}
H_f(s)=\sum_{m=1}^\infty\frac{c(m)}{m^s}
\end{align*}
we have
\begin{align}\label{26}
\sum_{m\leq x}c(m)\ll_\varepsilon x^{1/2+\vartheta+\varepsilon}
\end{align}
for every $\varepsilon>0$.

For simplicity, write
\begin{align*}
a_{f\times\tilde f}(m):=\sum_{m_1^{n-1}m_2^{n-2}\cdots m_{n-1}=m}|A(m_1,...,m_{n-1})|^2
\end{align*}
for the Dirichlet series coefficients of $L(s,f\times\widetilde f)$. By the properties of Dirichlet convolution together with (\ref{synttarit}) we have
\begin{align*}
|A(m,1,...,1)|^2=\sum_{d|m}c(d)a_{f\times\tilde f}\left(\frac md\right).
\end{align*}
Therefore
\begin{align*}
\sum_{m\leq x}|A(m,1,...,1)|^2&=\sum_{m\leq x}\sum_{d|m}c(d)a_{f\times\widetilde f}\left(\frac md\right)\\
&=\sum_{d\leq x}\sum_{\ell\leq\frac xd}c(d)a_{f\times\widetilde f}(\ell)\\
&\sim r_f\cdot x\sum_{d\leq x}\frac{c(d)}d
\end{align*}
by using (\ref{gl(n)rankin-selberg}). Combining this with the observation, which follows from (\ref{26}), the fact that $\vartheta\leq 1/2-1/(n^2+1)$, and partial summation,
\begin{align*}
\sum_{d>x}\frac{c(d)}d&\ll_\varepsilon\int_x^\infty\frac{t^{1/2+\vartheta+\varepsilon}}{t^2}\mathrm d t\\
&\ll_\varepsilon x^{-1/2+\vartheta+\varepsilon}\\
&\ll_\varepsilon x^{-1/(n^2+1)+\varepsilon}
\end{align*}
it follows that
\begin{align*}
\sum_{m\leq x}|A(m,1,...,1)|^2&\sim r_f\cdot x\sum_{d\leq x}\frac{c(d)}d\\
&=r_f\cdot x\sum_{d=1}^\infty\frac{c(d)}d+O\left(x\sum_{d>x}\frac{c(d)}d\right)\\
&\sim r_f\cdot H_f(1)\cdot x.
\end{align*}
This completes the proof. 
\end{proof}
\noindent As a consequence of this, we can evaluate the sum
\begin{align*}
\sum_{m\leq X^\theta}\frac{|A(m,1,...,1)|^2}{m^{1+1/n}}\sin^2\left(\frac{\pi\sqrt[n]m}L\right),
\end{align*}
where $0<\theta\leq 1$ is fixed, asymptotically.

By using partial summation we have
\begin{align}\label{2510}
&\sum_{m\leq X^\theta}\frac{|A(m,1,...,1)|^2}{m^{1+1/n}}\sin^2\left(\frac{\pi\sqrt[n]m}L\right)\nonumber\\
&=\left(1+\frac1n\right)\int_1^{X^\theta}\left(\sum_{m\leq x}|A(m,1,...,1)|^2\right)\frac{\sin^2\left(\frac{\pi\sqrt[n]x}L\right)}{x^{2+1/n}}\,\mathrm d x\nonumber\\
&-\frac{2\pi}n\cdot\frac1L\int_1^{X^\theta}\left(\sum_{m\leq x}|A(m,1,...,1)|^2\right)\frac{\sin\left(\frac{\pi\sqrt[n]x}L\right)\cos\left(\frac{\pi\sqrt[n]x}L\right)}{x^2}\,\mathrm d x+O\left(\frac1{X^{\theta/n}}\right),
\end{align}
where the error term comes from the substitution term by trivial estimation. The first term is, by using Theorem \ref{squaresum} and a simple change of variables, asymptotically
\begin{align*}
&\sim\left(1+\frac1n\right)\cdot r_f\cdot H_f(1)\int_1^{X^\theta}\frac1{x^{1+1/n}}\sin^2\left(\frac{\pi\sqrt[n]x}L\right)\,\mathrm d x\nonumber\\
&\sim\left(1+\frac1n\right)\cdot r_f\cdot H_f(1) \int_{1/L}^{X^{\theta/n}/L}\frac1{(yL)^{n+1}}\sin^2(\pi y)\cdot L^n n y^{n-1}\,\mathrm d y\nonumber\\
&\sim\frac{r_f\cdot H_f(1)\cdot(n+1)}L\int_{1/L}^{X^{\theta/n}/L}\frac{\sin^2(\pi y)}{y^2}\,\mathrm d y\nonumber\\
&\sim\frac{r_f\cdot H_f(1)\cdot(n+1)}L\cdot\frac{\pi^2}2,
\end{align*}
where the last estimate follows from the identity
\begin{align*}
\int_0^{\infty}\frac{\sin^2(\pi y)}{y^2}\,\mathrm d y=\frac{\pi^2}2
\end{align*}
together with the estimates
\begin{align*}
&\int_0^{1/L}\frac{\sin^2(\pi y)}{y^2}\,\mathrm d y\ll\frac1L,\\
&\int_{X^{\theta/n}/L}^\infty\frac{\sin^2(\pi y)}{y^2}\,\mathrm d y\ll\frac L{X^{\theta/n}}
\end{align*}
provided that $L=L(X)\longrightarrow\infty$ as $X\longrightarrow\infty$ and $L\ll_\varepsilon X^{\theta/n-\varepsilon}$ for some fixed $\varepsilon>0$.

An analogous computation shows that the second term on the right-hand side of (\ref{2510}) is
\begin{align*}
\sim -\frac{\pi^2}2\cdot \frac{r_f\cdot H_f(1)}L.
\end{align*} 
So, altogether
\begin{align}\label{sumasymptotics}
\sum_{m\leq X^\theta}\frac{|A(m,1,...,1)|^2}{m^{1+1/n}}\sin^2\left(\frac{\pi\sqrt[n]m}L\right)\sim\frac{r_f\cdot H_f(1)\cdot n}L\cdot\frac{\pi^2}2.
\end{align}

\section{Proof of Theorem 1}

Our proof follows the argument of Lester \cite{Lester}. Most of the steps are analogous, but we present the details for the sake of completeness. 

Let $f$ be a Hecke-Maass cusp form for the group $\mathrm{SL}(n,\mathbb Z)$ with Hecke eigenvalues $A(m,1,...,1)$. Let $0<\theta\leq 1$ and define 
\begin{align*}
P(x;\theta):=\frac{x^{1/2-1/2n}}{\pi\sqrt n}\sum_{m\leq X^\theta}\frac{A(1,...,1,m)}{m^{1/2+1/2n}}\cos\left(2\pi n\sqrt[n]{mx}+\frac{(n-3)}4\pi\right)
\end{align*}
for $X\leq x\leq 2X$.

Let us write
\begin{align*}
E(x;\theta):=\left(\sum_{m\leq x}A(m,1,...,1)\right)-P(x;\theta). 
\end{align*}
We remark that arguments similar to those in Section $7$ of \cite{Jaasaari-Vesalainen} show that
\begin{align}\label{pointwise}
E(x;\theta)\ll_\varepsilon x^{1-(1+\theta)/n+\vartheta+\varepsilon},
\end{align}
where $\vartheta$ is the exponent towards the Ramanujan-Petersson conjecture. The pointwise bound (\ref{pointwise}) is too weak to establish Theorem 1 but it will be shown that on average $E(x+x^{1-1/n}/L;\theta)-E(x;\theta)$ is much smaller than what this bound implies, of course under certain assumptions. 

The proof has three main steps. The first two are formulated in the following propositions. The first one evaluates the mean square of the expected main term for the sums of Hecke eigenvalues over a short interval $[x,x+x^{1-1/n}/L]$ for a suitable $L=L(X)$. 

\begin{proposition}\label{differenceofmainterms}
Let $f$ be a Hecke-Maass cusp form for the group $\mathrm{SL}(n,\mathbb Z)$ normalised so that $A(1,...,1)=1$. Let $0\leq\theta\leq 1$ and suppose that $2\leq L\ll_\varepsilon X^{1/(n(n-1))-\varepsilon}$ for some small fixed $\varepsilon>0$. Then we have
\begin{align*}
\frac1X\int_X^{2X}\left|P\left(x+\frac{x^{1-1/n}}L;\theta\right)-P(x;\theta)\right|^2\,\mathrm d x\sim\frac{X^{1-1/n}}L\cdot C_f,
\end{align*}
where $C_f$ is as in (\ref{constant}).
\end{proposition}

\noindent The other proposition shows that on average $P(x;\theta)$ is a sufficiently good approximation for the sum of Hecke eigenvalues $A(m,1,...,1)$ up to $x$ under the assumption of the generalised Lindel\"of hypothesis and a weak version of the Ramanujan-Petersson conjecture. This is better than the pointwise upper bounds for the error term one gets from the relevant Voronoi summation formula. 

\begin{proposition}\label{erroronaverage}
Let $f$ be a Hecke-Maass cusp form for the group $\mathrm{SL}(n,\mathbb Z)$ normalised so that $A(1,...,1)=1$. Suppose that $0<\theta<1/(n-1+2n\vartheta)$, where $\vartheta$ is the exponent towards the Ramanujan-Petersson conjecture, and assume also that $\vartheta<1/2-1/n$. Furthermore, suppose that the generalised Lindel\"of hypothesis for the $L$-function attached to the underlying Maass cusp form holds in the $t$-aspect. Then we have
\begin{align*}
&\frac1X\int_X^{2X}\left|E\left(x+\frac{x^{1-1/n}}L;\theta\right)-E(x;\theta)\right|^2\,\mathrm d x\ll_\varepsilon X^{1-(1+\theta)/n+\varepsilon}
\end{align*}
for every $\varepsilon>0$.
\end{proposition}

\begin{remark}
Notice that this bound is superior compared to the pointwise bound $\ll_\varepsilon X^{2-2(1+\theta)/n+2\vartheta+\varepsilon}$ which follows from (\ref{pointwise}).
\end{remark}

\noindent Once these have been established, the proof can be completed as follows. For now, let $\varepsilon>0$ be small but fixed. For notational simplicity, we set
\begin{align*}
S(x,L):=\sum_{x\leq m\leq x+x^{1-1/n}/L}A(m,1,...,1)
\end{align*}
and
\begin{align*}
Q(x,L,\theta):=P\left(x+\frac{x^{1-1/n}}L;\theta\right)-P(x;\theta).
\end{align*}
Then by making use of the elementary identity 
\begin{align*}
|S|^2=|Q|^2+|S-Q|^2+2\Re\left(Q(\overline S-\overline Q)\right)
\end{align*}
we obtain
\begin{align*}
&\frac1X\int_X^{2X}\left|\sum_{x\leq m\leq x+x^{1-1/n}/L}A(m,1,...,1)\right|^2\,\mathrm d x\\
&=\frac1X\int_X^{2X}|Q(x,L;\theta)|^2\,\mathrm d x+O\left(\frac1X\int_X^{2X}|S(x,L)-Q(x,L;\theta)|^2\,\mathrm d x\right)\\
&+O\left(\frac1X\int_X^{2X}|S(x,L)-Q(x,L;\theta)|\cdot|Q(x,L;\theta)|\,\mathrm d x\right).
\end{align*}
By Proposition \ref{differenceofmainterms}, the first term on the right-hand side is
\begin{align*}
\sim C_f\cdot\frac{X^{1-1/n}}L
\end{align*}
assuming $L\ll_\varepsilon X^{1/(n(n-1))-\varepsilon}$ and the second term is, say, $\ll_\varepsilon X^{1-(1+\theta)/n+\varepsilon/2}$ by Proposition \ref{erroronaverage} provided that $0<\theta<1/(n-1+2n\vartheta)$ and $\vartheta<1/2-1/n$. For the last term, an application of the Cauchy-Schwarz inequality yields
\begin{align*}
&\ll_\varepsilon\frac1X\cdot\left(\frac{X^{2-1/n}}L\right)^{1/2}\cdot(X^{2-(1+\theta)/n+\varepsilon/2})^{1/2}\\
&\ll_\varepsilon \frac{X^{1-(2+\theta)/2n+\varepsilon/4}}{L^{1/2}}. 
\end{align*}
Notice that this is smaller than the main term due to the assumption $L\ll_\varepsilon X^{\theta/n-\varepsilon}$. This completes the proof of Theorem 1. The next two subsections are devoted to the proofs of Propositions \ref{differenceofmainterms} and \ref{erroronaverage}. 

\subsection{Proof of Proposition \ref{differenceofmainterms}}
We start by writing
\begin{align*}
&P\left(x+\frac{x^{1-1/n}}L;\theta\right)-P(x;\theta)\\
&=\underbrace{P\left(x+\frac{x^{1-1/n}}L;\theta\right)-P\left(\left(\sqrt[n]{x}+\frac1{nL}\right)^n;\theta\right)}_{=:I_1(x,L;\theta)}+\underbrace{P\left(\left(\sqrt[n]{x}+\frac1{nL}\right)^n;\theta\right)-P(x;\theta)}_{=:I_2(x,L;\theta)}.
\end{align*}
The idea here is that $I_2(x,L;\theta)$ is easier to handle than the original difference and intuitively $I_1(x,L;\theta)$ should be small on average, which turns out to be the case. 
Then
\begin{align}\label{lastidentity}
&\frac1X\int_X^{2X}\left|P\left(x+\frac{x^{1-1/n}}L;\theta\right)-P(x;\theta)\right|^2\,\mathrm d x\\
&=\frac1X\int_X^{2X}\left|I_1(x,L;\theta)\right|^2\,\mathrm d x+\frac1X\int_X^{2X}\left|I_2(x,L;\theta)\right|^2\,\mathrm d x\nonumber\\
&\qquad\qquad\qquad+O\left(\frac1X\int_X^{2X}|I_1(x,L;\theta)I_2(x,L;\theta)|\,\mathrm d x\right).\nonumber
\end{align}
The proof of the proposition now proceeds by estimating the first two terms on the right-hand side separately. The second term is treated in Lemma \ref{I2} and the first term in Lemma \ref{I1}. The cross terms are handled by an application of the Cauchy-Schwarz inequality. Once we have shown that the contribution of the first term is $\ll_\varepsilon L^{-4}X^{1-1/n+(3-n)/(n(n-1))-\varepsilon}$ and the contribution of the second term is $\ll X^{1-1/n}/L$, it follows that the error term contributes
\begin{align*}
&\ll_\varepsilon \frac1X\left(\frac{X^{2-1/n+(3-n)/(n(n-1))-\varepsilon}}{L^4}\right)^{1/2}\left(\frac{X^{2-1/n}}{L}\right)^{1/2}\\
&\ll_\varepsilon \frac{X^{1-1/n+(3-n)/(2n(n-1))-\varepsilon/2}}{L^{5/2}},
\end{align*}
which is small enough as $n\geq 3$. 

\begin{lemma}\label{I2}
Suppose that $0<\theta\leq 1/(n-1)-\varepsilon$ for some small fixed $\varepsilon>0$. Then we have
\begin{align*}
\frac1X\int_X^{2X}\left|I_2(x,L;\theta)\right|^2\,\mathrm d x\sim\frac{X^{1-1/n}}L\cdot C_f. 
\end{align*}
\end{lemma}

\begin{proof}
To estimate the difference $I_2(x,L;\theta)$ we are reduced to understand terms of the form
\begin{align*}
&(x+\Xi)^{1/2-1/2n}\cos\left(2\pi n\sqrt[n]{mx}+\frac{2\pi\sqrt[n]{m}}L+\frac{(n-3)}4\pi\right)\\
&\qquad\qquad\qquad\qquad\qquad-x^{1/2-1/2n}\cos\left(2\pi n\sqrt[n]{mx}+\frac{(n-3)}4\pi\right),
\end{align*}
where $\Xi$ is given by the equation $x+\Xi=(\sqrt[n]x+1/nL)^n$. Now, the relevant observation is that
\begin{align*}
\left|(x+\Xi)^{1/2-1/2n}-x^{1/2-1/2n}\right|\asymp\left|\int_x^{x+\Xi}y^{-1/2-1/2n}\,\mathrm d y\right|\ll x^{-1/2-1/2n}\cdot\Xi. 
\end{align*}
But by the binomial theorem we have
\begin{align*}
\Xi\ll \frac{x^{(n-1)/n}}L.
\end{align*}
Therefore
\begin{align*}
\left|(x+\Xi)^{1/2-1/2n}-x^{1/2-1/2n}\right|\ll\frac{x^{1/2-3/2n}}L.
\end{align*}
This shows that 
\begin{align*}
&(x+\Xi)^{1/2-1/2n}\cos\left(2\pi n\sqrt[n]{mx}+\frac{2\pi\sqrt[n]{m}}L+\frac{(n-3)}4\pi\right)-x^{1/2-1/2n}\cos\left(2\pi n\sqrt[n]{mx}+\frac{(n-3)}4\pi\right)\\
&=x^{1/2-1/2n}\left(\cos\left(2\pi n\sqrt[n]{mx}+\frac{2\pi\sqrt[n]{m}}L+\frac{(n-3)}4\pi\right)-\cos\left(2\pi n\sqrt[n]{mx}+\frac{(n-3)}4\pi\right)\right)\\
&\quad+O\left(\frac{x^{1/2-3/2n}}L\cos\left(2\pi n\sqrt[n]{mx}+\frac{2\pi\sqrt[n]{m}}L+\frac{(n-3)}4\pi\right)\right).
\end{align*}
By using the formula for the difference of two cosines, $\cos(\alpha)-\cos(\beta)=-2\sin((\alpha+\beta)/2)\sin((\alpha-\beta)/2)$, it follows that
\begin{align*}
&I_2(x,L;\theta)\\
&=\frac{-2x^{1/2-1/2n}}{\pi\sqrt n}\sum_{m\leq X^\theta}\frac{A(1,...,1,m)}{m^{1/2+1/2n}}\sin\left(\frac{\pi\sqrt[n]{m}}L\right)\sin\left(2\pi n\sqrt[n]{m}\left(\sqrt[n]{x}+\frac1{2nL}\right)+\frac{(n-3)}4\pi\right)\\
&+O\left(\frac{x^{1/2-3/2n}}L\left|\sum_{m\leq X^\theta}\frac{A(1,...,1,m)}{m^{1/2+1/2n}}\cos\left(2\pi n\sqrt[n]{m}\left(\sqrt[n]{x}+\frac1{nL}\right)+\frac{(n-3)}4\pi\right)\right|\right)\\
&=:M(x,L;\theta)+R(x,L;\theta).
\end{align*}
Hence,
\begin{align}\label{225}
\frac1X\int_X^{2X}|I_2(x,L;\theta)|^2\,\mathrm d x&=\frac1X\int_X^{2X}|M(x,L;\theta)|^2\,\mathrm d x+\frac1X\int_X^{2X}|R(x,L;\theta)|^2\,\mathrm d x\nonumber\\
&\quad+O\left(\frac1X\int_X^{2X}|M(x,L;\theta)R(x,L;\theta)|\,\mathrm d x\right).
\end{align}
The main term can be written as
\begin{align*}
M(x,L;\theta)=-\frac{x^{1/2-1/2n}}{\pi i\sqrt n}\left(\sum_{m\leq X^\theta} a_m^+e(n\sqrt[n]{mx})-\sum_{m\leq X^\theta}a_m^-e(-n\sqrt[n]{mx})\right),
\end{align*}
where
\begin{align*}
a_m^\pm:=\frac{A(1,...,1,m)}{m^{1/2+1/2n}}e\left(\pm\frac{\sqrt[n]m}{2L}\pm\frac{(n-3)}8\right)\sin\left(\frac{\pi\sqrt[n]m}L\right).
\end{align*}
Let us first evaluate the mean square of $M(x,L;\theta)$. Notice that
\begin{align}\label{heinakuu}
&|M(x,L;\theta)|^2=\frac{x^{1-1/n}}{n\pi^2}\left(\left|\sum_{m\leq X^\theta}a_m^+e(n\sqrt[n]{mx})\right|^2+\left|\sum_{m\leq X^\theta}a_m^-e(-n\sqrt[n]{mx})\right|^2\right)\nonumber\\
&\quad-\frac{2x^{1-1/n}}{n\pi^2}\Re\left(\left(\sum_{m\leq X^\theta}\overline{a_m^+}e(-n\sqrt[n]{mx})\right)\left(\sum_{m\leq X^\theta} a_m^-e(-n\sqrt[n]{mx})\right)\right).
\end{align}
We consider the first two terms on the right-hand side simultaneously as their treatment is identical due to the fact that $|a_m^+|=|a_m^-|$. By opening the absolute squares these split into diagonal and off-diagonal terms. By the first derivative test the non-diagonal terms give a contribution
\begin{align*}
&\ll X^{1-2/n}\sum_{\substack{1\leq m,\ell\leq X^\theta\\
m>\ell}}\frac{|a_m^+a_\ell^+|}{\sqrt[n]m-\sqrt[n]\ell}\\
&\ll X^{1-2/n}\sum_{\substack{1\leq m,\ell\leq X^\theta\\
m>\ell}}\frac{|a_m^+a_\ell^+|m^{1-1/n}}{|m-\ell|}\\
&\ll X^{1-2/n}X^{\theta(1-1/n)}\log X\sum_{1\leq m\leq X^\theta}|a_m^+|^2,
\end{align*}
where the last estimate follows from the elementary estimate $ab\ll a^2+b^2$. 

The total contribution coming from the diagonal terms is
\begin{align*}
&\frac{(2^{2-1/n}-1)}{(2-1/n)n\pi^2}\left(\sum_{m\leq X^\theta}|a_m^+|^2+\sum_{m\leq X^\theta}|a_m^-|^2\right)X^{1-1/n}\\
&=\frac{2(2^{2-1/n}-1)}{(2-1/n)n\pi^2}\cdot X^{1-1/n}\sum_{m\leq X^\theta}\frac{|A(m,1,...,1)|^2}{m^{1+1/n}}\sin^2\left(\frac{\pi\sqrt[n]m}L\right).
\end{align*}
For the third term in (\ref{heinakuu}) we observe that it can be estimated similarly by using the first derivative test as the off-diagonal terms above. Therefore it follows that 
\begin{align*}
&\frac1X\int_X^{2X}|M(x,L;\theta)|^2\,\mathrm d x\\
&=\frac{2}{n\pi^2}\cdot\frac{2^{2-1/n}-1}{2-1/n}X^{1-1/n}\sum_{m\leq X^\theta}\frac{|A(m,1,...,1)|^2}{m^{1+1/n}}\sin^2\left(\frac{\pi\sqrt[n]m}L\right)\\
&\qquad\quad+O\left(X^{1-2/n+\theta(1-1/n)}\log X\sum_{m\leq X^\theta}|a_m^+|^2\right).
\end{align*}
By using (\ref{sumasymptotics}) we infer that
\begin{align*}
\sum_{m\leq X^\theta}|a_m^+|^2=\sum_{m\leq X^\theta}\frac{|A(m,1,...1)|^2}{m^{1+1/n}}\sin^2\left(\frac{\pi\sqrt[n]m}L\right)\sim\frac{r_f\cdot H_f(1)\cdot n}L\cdot\frac{\pi^2}2.
\end{align*}
The assumption $\theta<1/(n-1)-\varepsilon$ guarantees that the error term is smaller than the main term. The mean square of the remainder term $R(x,L,\theta)$ is treated similarly: it is
\begin{align*}
\ll \frac1X\cdot\frac1{L^2}\cdot X^{2-3/n}\sum_{m\leq X^\theta}\frac{|A(m,1,...,1)|^2}{m^{1+1/n}}\ll\frac{X^{1-3/n}}{L^2}
\end{align*}
by using (\ref{final-bound}).

Finally, cross-terms in (\ref{225}) are handled by a single application of the Cauchy-Schwarz inequality; they contribute
\begin{align*}
&\ll_\varepsilon\frac1X\left(\frac{X^{2-3/n}}{L^2}\right)^{1/2}\left(\frac{X^{2-1/n+\varepsilon}}L\right)^{1/2}\\
&\ll_\varepsilon \frac{X^{1-2/n+\varepsilon/2}}{L^{3/2}},
\end{align*}
which is smaller than the main term if $\varepsilon$ is small enough in terms of $n$. This completes the proof of the lemma. 
\end{proof}
\noindent Next, we will compute the mean square of $I_1(x,L;\theta)$.
\begin{lemma}\label{I1}
Assume that $0<\theta<1/(n-1)-\varepsilon$ for some fixed $\varepsilon>0$. Then we have 
\begin{align*}
\frac1X\int_X^{2X}\left|I_1(x,L;\theta)\right|^2\mathrm d x\ll_\varepsilon \frac{X^{1-1/n+(3-n)/(n(n-1))-\varepsilon}}{L^4}+\frac{X^{1-5/n}}{L^4}.
\end{align*}
\end{lemma}

\begin{proof}
For simplicity, we set
\begin{align*}
x_1:=\sqrt[n]x+\frac1{nL}\qquad\text{and}\qquad x_2:=\left(x+\frac{x^{1-1/n}}L\right)^{1/n}.
\end{align*} 
Then 
\begin{align*}
I_1(x,L;\theta)=x_2^{(n-1)/2}\sum(x_2)-x_1^{(n-1)/2}\sum(x_1),
\end{align*}
where we have set
\begin{align*}
\sum(x):=\frac1{\pi\sqrt n}\sum_{m\leq X^\theta}\frac{A(1,...,1,m)}{m^{1/2+1/2n}}\cos\left(2\pi n x\sqrt[n]m+\frac{(n-3)}4\pi\right).
\end{align*}
By the triangle inequality we get
\begin{align*}
&\left|x_2^{(n-1)/2}\sum(x_2)-x_1^{(n-1)/2}\sum(x_1)\right|\\
&=\left|\left(x_1^{(n-1)/2}-x_2^{(n-1)/2}\right)\sum(x_1)+x_2^{(n-1)/2}\left(\sum(x_1)-\sum(x_2)\right)\right|\\
&\ll\left|x_1^{(n-1)/2}-x_2^{(n-1)/2}\right|\cdot\left|\sum(x_1)\right|+x_2^{(n-1)/2}\left|\sum(x_1)-\sum(x_2)\right|.
\end{align*}
By the mean value theorem we have
\begin{align*}
\left|x_1^{(n-1)/2}-x_2^{(n-1)/2}\right|\asymp\int_{x_1}^{x_2}t^{(n-3)/2}\,\mathrm d t\asymp|x_1-x_2|X^{(n-3)/2n}.
\end{align*}
For the second term we observe that
\begin{align*}
&\sum(x_1)-\sum(x_2)\\
&=\frac1{\pi\sqrt n}\sum_{m\leq X^\theta}\frac{A(1,...,1,m)}{m^{1/2+1/2n}}\left(\cos\left(2\pi n\sqrt[n]m x_1+\frac{(n-3)}4\pi\right)-\cos\left(2\pi n\sqrt[n]m x_2+\frac{(n-3)}4\pi\right)\right)\\
&\asymp\sum_{m\leq X^\theta}\frac{A(1,...,1,m)}{m^{1/2+1/2n}}\cdot m^{1/n}|x_1-x_2|\\
&=\sum_{m\leq X^\theta}\frac{A(1,...,1,m)}{m^{1/2-1/2n}}|x_1-x_2|,
\end{align*}
as
\begin{align*}
&\left|\cos\left(2\pi n\sqrt[n]m x_1+\frac{(n-3)}4\pi\right)-\cos\left(2\pi n\sqrt[n]m x_2+\frac{(n-3)}4\pi\right)\right|\\
&\asymp \left|m^{1/n}\int_{x_1}^{x_2}\sin\left(2\pi n\sqrt[n]m t+\frac{(n-3)}4\pi\right)\mathrm d t\right|\\
&\ll m^{1/n}|x_1-x_2|.
\end{align*}
Thus we have 
\begin{align*}
&\left|x_1^{(n-1)/2}\sum(x_1)-x_2^{(n-1)/2}\sum(x_2)\right|\\
&\ll |x_1-x_2|\left(X^{1/2-1/2n}\sum_{m\leq X^\theta}\frac{A(1,...,1,m)}{m^{1/2-1/2n}}+X^{(n-3)/2n}\left|\sum(x_1)\right|\right).
\end{align*}
But as, say, 
\begin{align*}
\sum_{m\leq X^\theta}\frac{A(1,...,1,m)}{m^{1/2-1/2n}}\ll_\varepsilon X^{\theta(1/2+1/2n)+\varepsilon/2n}
\end{align*}
by partial summation, and 
\begin{align*}
|x_1-x_2|&\asymp\left|\int_{x+\Xi}^{x+x^{1-1/n}/L}t^{1/n-1}\,\mathrm d t\right|\\
&\ll \left|\frac{x^{1-1/n}}L-\Xi\right|(x+\Xi)^{1/n-1}\\
&\ll\frac{x^{1-2/n}}{L^2}\cdot X^{1/n-1}\\
&\ll\frac1{L^2X^{1/n}},
\end{align*}
it follows that this can be further estimated to be
\begin{align*}
\ll_\varepsilon \frac1{L^2X^{1/n}}\left(X^{1/2-1/2n+\theta(1/2+1/2n)+\varepsilon/2n}+X^{(n-3)/2n}\left|\sum(x_1)\right|\right).
\end{align*}
By using the inequality $ab\ll a^2+b^2$ we infer
\begin{align*}
\frac1X\int_X^{2X}|I_1(x,L;\theta)|^2\,\mathrm d x&\ll_\varepsilon\frac{X^{1-1/n+\theta(1+1/n)+\varepsilon/n}}{L^4 X^{2/n}}+\frac{X^{(n-3)/n}}{L^4X^{2/n}}\cdot\frac1X\int_X^{2X}\left|\sum(x_1)\right|^2\,\mathrm d x\\
&\ll_\varepsilon\frac{X^{1-3/n+\theta(1+1/n)+\varepsilon/n}}{L^4}+\frac{X^{1-5/n}}{L^4}.
\end{align*}
The claim follows from this by recalling that $\theta<1/(n-1)-\varepsilon$. In the last step we have used the fact that 
\begin{align*}
\frac1X\int_X^{2X}\left|\sum(x_1)\right|^2\,\mathrm d x\ll 1.
\end{align*}
This follows by opening the absolute square and integrating termwise. The off-diagonal contributes 
\begin{align*}
&\ll_\varepsilon X^{-1/n+\theta(1-1/n)+\varepsilon(n-1)/n}\sum_{m\leq X^\theta}\frac{|A(m,1,...,1)|^2}{m^{1+1/n}}\\
&\ll_\varepsilon X^{-1/n+\theta(1-1/n)+\varepsilon(n-1)/n}\\
&\ll 1
\end{align*}
by using the first derivative test and the assumption $\theta<1/(n-1)-\varepsilon$. The diagonal term is obviously 
\begin{align*}
\ll \sum_{m\leq X^\theta}\frac{|A(m,1,...,1)|^2}{m^{1+1/n}}\ll 1.
\end{align*}
This completes the proof.    
\end{proof}
\noindent Now, as Lemmas \ref{I2} and \ref{I1} are proved, the proof of Proposition \ref{differenceofmainterms} is completed by the discussion above. \qed

\subsection{Proof of Proposition \ref{erroronaverage}}
Recall that
\begin{align*}
E(x;\theta)=\left(\sum_{m\leq x}A(m,1,...,1)\right)-P(x;\theta).
\end{align*}
Throughout the proof, let $\varepsilon>0$ be small but fixed. We start by simply estimating
\begin{align*}
&\frac1X\int_X^{2X}\left|E\left(x+\frac{x^{1-1/n}}L;\theta\right)-E(x;\theta)\right|^2\,\mathrm d x\\
&\ll\frac1X\int_X^{2X}\left|E\left(x+\frac{x^{1-1/n}}L;\theta\right)\right|^2\,\mathrm d x+\frac1X\int_X^{2X}\left|E(x;\theta)\right|^2\,\mathrm d x.
\end{align*}
Analysis of both terms on the right-hand side is similar and hence we concentrate on the latter term
\begin{align*}
\frac1X\int_X^{2X}\left|E(x;\theta)\right|^2\,\mathrm d x.
\end{align*}
As usual, the starting point is the truncated Perron's formula which gives, for a small enough fixed $\delta>0$,
\begin{align*}
\sum_{m\leq x}A(m,1,...,1)=\frac1{2\pi i}\int_{1+\delta-iX}^{1+\delta+iX}L(s,f)x^s\frac{\mathrm d s}s+O(X^{\vartheta+\varepsilon})
\end{align*}
uniformly for $X\leq x\leq 2X$. 
 
The error term is admissible as we assume that $\vartheta<1/2-1/n$. We shift the line segment of integration first to the line $\sigma=1/2$. The Phragm\'en-Lindel\"of principle tells that in the strip $1/2\leq\sigma\leq 1+\delta$ the estimate of the form $L(s,f)\ll_\varepsilon\langle t\rangle^\varepsilon$ holds under the assumption of the generalised Lindel\"of hypothesis. By using this, the vertical line segments from the shift contribute
\begin{align*}
&\ll\int_{1/2}^{1+\delta}L(\sigma\pm iX,f)x^{\sigma+iX}\frac{\mathrm d \sigma}{\sigma\pm iX}\\
&\ll_\varepsilon X^{\varepsilon-1/2}+X^{\delta+\varepsilon}\\
&\ll_\varepsilon X^{\delta+\varepsilon}.
\end{align*}
It follows that
\begin{align*}
&\sum_{m\leq x}A(m,1,...,1)=\frac1{2\pi i}\int_{1/2-iY}^{1/2+iY}L(s,f)x^s\frac{\mathrm d s}s+\frac1{2\pi i}\int_{1/2-iX}^{1/2-iY}L(s,f)x^s\frac{\mathrm d s}s\\
&\qquad\qquad\qquad\qquad+\frac1{2\pi i}\int_{1/2+iY}^{1/2+iX}L(s,f)x^s\frac{\mathrm d s}s+O\left(X^{\vartheta+\varepsilon/2}\right),
\end{align*}
uniformly for $X\leq x\leq 2X$, where $0<Y<X$ is a parameter chosen later.

Next we move the line segment of integration to the line $\sigma=-\delta$ in the first term on the right-hand side. Using the convexity bound $L(s,f)\ll \langle t\rangle^{(1+\delta-\sigma)n/2}$ in the vertical strip $-\delta\leq\sigma\leq 1+\delta$ together with the assumption that $L(1/2+it,f)\ll_\varepsilon\langle t\rangle^\varepsilon$ it follows that the vertical line segments contribute
\begin{align*}
&\ll\int_{-\delta}^{1/2}L(\sigma\pm iY,f)x^{\sigma+iY}\frac{\mathrm d \sigma}{\sigma\pm iY}\\
&\ll_\varepsilon Y^{\varepsilon-1}X^{1/2}+Y^{n/2+n\delta-1}X^{-\delta}\\
&\ll_\varepsilon X^{1/2-(1+\theta)/2n+\varepsilon/2}
\end{align*}
as we are going to choose $\delta$ so that $2\delta\theta\leq\varepsilon$, where the last estimate follows from the assumption on $\theta$ as we are going to choose $Y$ such that it satisfies $Y\asymp X^{(1+\theta)/n}$.

Next, we treat the term 
\begin{align*}
\frac1{2\pi i}\int_{-\delta-iY}^{-\delta+iY}L(s,f)x^s\frac{\mathrm d s}s.
\end{align*}
We are now in the position to apply the proof method of the Voronoi summation formula used in \cite{Jaasaari-Vesalainen} in the case $n=3$. Since we intend to apply Stirling's formula, we write
\begin{align*}
\frac1{2\pi i}\int_{-\delta-iY}^{-\delta+iY}L(s,f)\,x^s\,\frac{\mathrm ds}s
=\frac1{2\pi i}\left(\,\,\int_{-\delta-iY}^{-\delta-i\Lambda}+\int_{-\delta+i\Lambda}^{-\delta+iY}\,\,\right)L(s,f)\,x^s\,\frac{\mathrm ds}s+O_{\delta,\Lambda}(1),
\end{align*}
where $\Lambda:=1+2\max_{1\leq j\leq n}\{|\lambda_j(\nu)|,|\widetilde\lambda_j(\nu)|\}$. Now we may apply the functional equation of Godement--Jacquet $L$-functions (\ref{functionalequation}), interchange the order of integration, and summation and apply Stirling's formula to get
\begin{align}\label{vihoviimeinen}
&\frac1{2\pi i}\left(\int_{-\delta-iY}^{-\delta-i\Lambda}+\int_{-\delta+i\Lambda}^{-\delta+iY}\right)L(s,f)\,x^s\,\frac{\mathrm ds}s\nonumber\\
&=\frac1{2\pi i}\left(\int_{-\delta-iY}^{-\delta-i\Lambda}+\int_{-\delta+i\Lambda}^{-\delta+iY}\right)\pi^{ns-n/2}\,\frac{G(1-s,\widetilde f)}{G(s,f)}\,L(1-s,\widetilde f)\,x^s\,\frac{\mathrm ds}s\nonumber\\
&=\frac1{2\pi}\sum_{m=1}^\infty\frac{A(1,\ldots,1,m)}m\nonumber\\
&\cdot\left(\int_{-\delta-iY}^{-\delta-i\Lambda}+\int_{-\delta+i\Lambda}^{-\delta+iY}\right)\frac1i\,\pi^{ns-n/2}\,n^{ns-n/2}\,\frac{\Gamma\!\left(\frac{1-ns}2\right)}{\Gamma\!\left(\frac{ns-(n-1)}2\right)}\left(1+O(\left|s\right|^{-1})\right)\,m^s\,x^s\,\frac{\mathrm ds}s.
\end{align}
In the region of integration the quotient of $\Gamma$-factors is $\ll t^{n/2+n\delta}$ by Stirling's formula, and so the series corresponding to the $O$-term can be estimated to be
\begin{align*}
&\ll\sum_{m=1}^\infty\frac{\left|A(1,\ldots,1,m)\right|}{m^{1+\delta}}
\int_\Lambda^Yt^{n/2+n\delta}\,t^{-1}\,x^{-\delta}\,\frac{\mathrm dt}t\\
&\ll x^{-\delta}\,Y^{n/2+n\delta-1}\\
&\ll_\varepsilon X^{\varepsilon/2}\,Y^{n/2-1}\\
&\ll_\varepsilon X^{1/2-(1+\theta)/n+\theta/2+\varepsilon/2}\\
&\ll_\varepsilon X^{1/2-(1+\theta)/2n+\varepsilon/2},
\end{align*}
provided that $2\theta\delta\leqslant\varepsilon$, by using (\ref{final-bound}), the assumption on $\theta$, and the fact that $Y\asymp X^{(1+\theta)/n}$.

We are going to transform rest of the integral in (\ref{vihoviimeinen}) further by making a simple change of variables to rewrite it as 
\begin{align*}
2\Re\left(\int_\Lambda^Y
(\pi n)^{n(-\delta+it)-n/2}\,\frac{\Gamma\!\left(\frac{1-n(-\delta+it)}2\right)}{\Gamma\!\left(\frac{n(-\delta+it)-(n-1)}2\right)}\left(mx\right)^{-\delta+it}\frac{\mathrm dt}{-\delta+it}\right).
\end{align*}
Using the elementary fact that
\begin{align*}
\frac1{-\delta+it}=\frac1{it}+O(t^{-2}),
\end{align*}
this equals
\begin{align*}
2\Re\left(\int_\Lambda^Y
(\pi n)^{n(-\delta+it)-n/2}\,\frac{\Gamma\!\left(\frac{1-n(-\delta+it)}2\right)}{\Gamma\!\left(\frac{n(-\delta+it)-(n-1)}2\right)}\left(mx\right)^{-\delta+it}\frac{\mathrm dt}{it}\right)
+O(X^\varepsilon\,Y^{n/2-1}).
\end{align*}

By Stirling's formula,
\begin{align*}
\frac{\Gamma\!\left(\frac{1-ns}2\right)}{\Gamma\!\left(\frac{ns-(n-1)}2\right)}
=\left(\frac{nt}2\right)^{n/2-n\sigma}\,\exp\!\left(-int\log\frac{nt}2+int+\frac{\pi ni}4\right)\left(1+O(t^{-1})\right).
\end{align*}
Substituting this back to the last integral, and observing that the terms coming from the $O(t^{-1})$-term contribute $\ll_\varepsilon X^\varepsilon\,Y^{n/2-1}$, it takes the form
\begin{multline*}
2\,(2\pi)^{-n/2}\left(2^n\pi^nmx\right)^{-\delta}\\
\cdot\Re\left(\int\limits_\Lambda^Y
t^{n/2+n\delta-1}\,\exp\!\left(-int\log\frac t{2\pi}+it\log(mx)+int+\frac{\pi (n-2)i}4\right)
\mathrm dt\right).
\end{multline*}

\noindent The derivative of the phase is, up to a constant, given by
\begin{align*}
-n\log t+\log(2^n\,\pi^n\,m\,x),
\end{align*}
and so the integrand has a unique saddle point at \[t=2\pi\,(mx)^{1/n}.\]
We will choose $Y$ to be
\begin{align*}
Y:=2\pi\left(\!\left(X^\theta+\frac12\right)x\right)^{1/n},
\end{align*}
so that for the terms $m>X^\theta$ the integrands have no saddle-points and are therefore oscillating. 

First, we treat these high-frequency terms with $m>X^\theta$. By using the fact that $t\leq Y$, the derivative of the phase in the corresponding integrals is
\begin{align*}
\log\frac{2^n\pi^nmx}{t^n}\gg\log\frac{2^n\pi^nmx}{Y^n}
=\log\frac m{X^\theta+\frac12},
\end{align*}
and so, by the first derivative test and (\ref{final-bound}), they contribute
\begin{align*}
&\ll
X^{-\delta}\,Y^{n/2+n\delta-1}\sum_{m>X^\theta}\frac{\left|A(1,\ldots,1,m)\right|}{m^{1+\delta}}\cdot\frac1{\log\frac m{X^\theta+\frac12}}\\
&\ll
X^{-\delta}\,Y^{n/2+n\delta-1}\sum_{X^\theta<m\leqslant2X^\theta}\frac{\left|A(1,\ldots,1,m)\right|}{m^{1+\delta}\left(\frac m{X^\theta+\frac12}-1\right)}+X^{-\delta}\,Y^{n/2+n\delta-1}\\
&\ll_\varepsilon X^{(1/2+\delta-1/n)(1+\theta)-\delta+\theta\vartheta+\varepsilon/2}+X^{(1/2+\delta-1/n)(1+\theta)-\delta}\\
&\ll_\varepsilon X^{1/2-(1+\theta)/2n+\varepsilon/2},
\end{align*}
where the elementary fact that $\log x\gg x-1$ for $x\in\left]1,2\right[$ is used in the second estimate, in the penultimate step we have used the fact that $Y\asymp X^{(1+\theta)/n}$, in the last estimate we have used that $\theta<1/(n-1+2n\vartheta)$, and finally we have bounded the sum trivially by using the absolute values:
\begin{align*}
&\sum_{X^\theta<m\leqslant2X^\theta}\frac{\left|A(1,\ldots,1,m)\right|}{m^{1+\delta}\left(\frac m{X^\theta+\frac12}-1\right)}
\ll X^{\theta\vartheta-\theta\delta}\sum_{X^\theta<m\leqslant2X^\theta}\frac1{m-X^\theta-\frac12}\ll_\varepsilon X^{\theta\vartheta+\varepsilon/2}.
\end{align*}
Next, we will deal with the low-frequency terms, that is, terms with $m\leqslant X^\theta$. First, we extend the integrals over the line segments $[-\delta-iY,-\delta-i\Lambda]$ and $[-\delta+i\Lambda,-\delta+iY]$ to be over the whole line segment connecting $-\delta-iY$ to $-\delta+iY$ with an error $O_{\delta,\Lambda}(1)$. Similarly, we may replace the factor $s^{-1}$ by $(s+\Lambda)^{-1}$ with the error $O(X^{\varepsilon/2}\,Y^{n/2-1})$. Thus, the terms that we are left to deal with are
\begin{align*}
(n\pi)^{-n/2}\sum_{m\leqslant X^\theta}\frac{A(1,\ldots,1,m)}m\cdot
\frac1{2\pi i}\int\limits_{-\delta-iY}^{-\delta+iY}\frac{\Gamma\!\left(\frac{1-ns}2\right)}{\Gamma\!\left(\frac{ns-(n-1)}2\right)}\left(\pi^nn^nmx\right)^s\frac{\mathrm ds}{s+\Lambda}.
\end{align*}
The main terms come from using Corollary \ref{key-corollary} on these integrals with the choice $y=\pi^nn^nmx$. These main terms are given by
\begin{align*}
\frac{x^{1/2-1/2n}}{\pi\sqrt n}\,\sum_{m\leqslant X^\theta}
\frac{A(1,\ldots,1,m)}{m^{1/2+1/2n}}\,\cos\left(2n\pi(mx)^{1/n}+\frac{(n-3)}4\,\pi\right)=P(x;\theta).
\end{align*}
Note that the condition $y<(nY/2)^n$ is satisfied by the choice of $Y$. The contribution coming from the error terms of Corollary \ref{key-corollary} can be estimated as follows by using partial summation together with (\ref{gl(n)rankin-selberg}) and recalling the fact that $Y\asymp X^{(1+\theta)/n}$:
\begin{align*}
&\sum_{m\leqslant X^\theta}\frac{\left|A(1,\ldots,1,m)\right|}m
\left((mx)^{1/2-1/2n-1/n}+Y^{n/2-1+n\delta}+Y^{n/2-1}\frac{1}{\log\frac{Y^{n}}{2^n\pi^nmx}}\right)\\
&\ll_\varepsilon X^{(1/2-3/2n)(1+\theta)+\varepsilon/2}\\
&\qquad+X^{(1/2-1/n+\delta)(1+\theta)}+X^{(1/2-1/n)(1+\theta)}\sum_{m\leqslant X^ \theta}\frac{\left|A(1,\ldots,1,m)\right|}m\,\frac{1}{\log\frac{X^\theta+\frac12}m}\\
&\ll_\varepsilon X^{(1/2-3/2n)(1+\theta)+\varepsilon/2}\,+X^{(1/2-1/n)(1+\theta)+\varepsilon/4}\,\left(1+\sum_{m\leqslant X^\theta}\frac{\left|A(1,\ldots,1,m)\right|}{X^\theta+\frac12-m}\right)\\
&\ll_\varepsilon X^{1/2-(1+\theta)/2n+\varepsilon/2}
\end{align*}
if $2\delta(1+\theta)\leq\varepsilon$, where the last estimate follows simply by using the absolute values as before. Therefore we have shown that 
\begin{align*}
\frac1{2\pi i}\int_{-\delta-iY}^{-\delta+iY}L(s,f)x^s\frac{\mathrm d s}s=P(x;\theta)+O\left(X^{1/2-(1+\theta)/2n+\varepsilon/2}\right)
\end{align*}
in the range $0<\theta<1/(n-1+2n\vartheta)$ assuming $\vartheta< 1/2-1/n$. 
Furthermore, it follows that 
\begin{align*}
E(x;\theta)=\int_{1/2-iX}^{1/2-iY}L(s,f)x^s\frac{\mathrm d s}s+\int_{1/2+iY}^{1/2+iX}L(s,f)x^s\frac{\mathrm d s}s+O\left(X^{1/2-(1+\theta)/2n+\varepsilon/2}\right)
\end{align*}
for $0<\theta<1/(n-1+2n\vartheta)$ and $\vartheta<1/2-1/n$.

From this it follows that 
\begin{align*}
&\frac1X\int_X^{2X}\left|E(x;\theta)\right|^2\,\mathrm d x\\
&\ll_\varepsilon X^{1-(1+\theta)/n+\varepsilon}+\frac1X\int_X^{2X}\left|\int_{1/2-iX}^{1/2-iY}L(s,f)x^s\frac{\mathrm d s}s+\int_{1/2+iY}^{1/2+iX}L(s,f)x^s\frac{\mathrm d s}s\right|^2\,\mathrm d x
\end{align*}
for $\theta$ and $\vartheta$ in the same ranges as before.

Hence, we are now reduced to study 
\begin{align*}
\frac1X\int_X^{2X}\left|\int_{1/2-iX}^{1/2-iY}L(s,f)x^s\frac{\mathrm d s}s+\int_{1/2+iY}^{1/2+iX}L(s,f)x^s\frac{\mathrm d s}s\right|^2\,\mathrm d x.
\end{align*}
Let us fix a smooth compactly supported non-negative weight function $w$ majorising the characteristic function of the interval $[1,2]$.

Now we simply compute: 
\begin{align*}
&\frac1X\int_X^{2X}\left|\int_{1/2-iX}^{1/2-iY}L(s,f)x^s\frac{\mathrm d s}s+\int_{1/2+iY}^{1/2+iX}L(s,f)x^s\frac{\mathrm d s}s\right|^2\,\mathrm d x\\
&\leq\frac1X\int_{\mathbb R}\left|\int_{1/2-iX}^{1/2-iY}L(s,f)x^s\frac{\mathrm d s}s+\int_{1/2+iY}^{1/2+iX}L(s,f)x^s\frac{\mathrm d s}s\right|^2 w\left(\frac xX\right)\,\mathrm d x\\
&\ll\frac1{ X}\int_{\mathbb R}\left|\int_{1/2+iY}^{1/2+iX}L(s,f)x^s\frac{\mathrm d s}s\right|^2w\left(\frac xX\right)\,\mathrm d x\\
&=\frac1{X}\int_{\mathbb R}\int_{1/2+iY}^{1/2+iX}\int_{1/2+iY}^{1/2+iX}L(s_1,f)x^{s_1}\overline{L(s_2,f)}x^{\overline s_2}\frac{\mathrm d s_1}{s_1}\frac{\mathrm d \overline{s_2}}{\overline{s_2}}w\left(\frac xX\right)\,\mathrm d x\\
&=\frac1{X}\int_{\mathbb R}\int_Y^X\int_Y^XL\left(\frac12+it,f\right)x^{1/2+it}\overline{L\left(\frac12+iv,f\right)}x^{1/2-iv}\frac{\mathrm d t\mathrm d v}{\left(\frac12+it\right)\left(\frac12-iv\right)}w\left(\frac xX\right)\,\mathrm d x\\
&=\int_Y^X\int_Y^X\frac{L\left(\frac12+it,f\right)L\left(\frac12-iv,\widetilde f\right)}{\left(\frac12+it\right)\left(\frac12-iv\right)}X^{1+i(t-v)}\left(\int_{\mathbb R}x^{1+i(t-v)}w(x)\,\mathrm d x\right)\,\mathrm d t\,\mathrm d v.
\end{align*}
By repeated integration by parts we see that the inner integral is negligible (i.e. $\ll_A X^{-A}$ for any $A>0$) when $|t-v|\geq X^\eta$ for some fixed $\eta>0$. In the complementary range the inner integral is bounded. Using this we simply estimate that the remaining part of the integral is
\begin{align*}
&\ll_\eta X\int\int_{\substack{Y\leq t,v\leq X\\
|t-v|\leq X^\eta}}\frac{\left|L\left(\frac12+it,f\right)L\left(\frac12-iv,\widetilde f\right)\right|}{tv}\,\mathrm d t\,\mathrm d v\\
&\ll_{\eta,\varepsilon} X\int\int_{\substack{Y\leq t,v\leq X\\
|t-v|\leq X^\eta}}(tv)^{-1+\varepsilon n(1+\theta)/2}\,\mathrm d t\,\mathrm d v\\
&\ll_{\eta,\varepsilon} X^{1+\eta}Y^{-1+\varepsilon n(1+\theta)/2}\int_{Y-X^\eta}^{X+X^\eta}t^{-1+\varepsilon n(1+\theta)/2}\,\mathrm d t\\
&\ll_{\eta,\varepsilon} X^{1+\eta}Y^{-1+\varepsilon n(1+\theta)/2}\\
&\ll_{\eta,\varepsilon} X^{1-(1+\theta)/n+\eta+\varepsilon/2}, 
\end{align*}
where, in the third step, we have used the fact that, for fixed $t$, the parameter $v$ ranges over a set of measure $\asymp X^\eta$. The resulting upper bound is small enough if we choose $\eta=\varepsilon/2>0$. Here we have used the generalised Lindel\"of hypothesis in the second step and the fact that $Y\asymp X^{(1+\theta)/n}$ in the penultimate step. This finishes the proof. \qed 

\section{Proof of Theorem 2}

We start by observing that
\begin{align}\label{1372}
&\frac1X\int_X^{2X}\left|\sum_{x\leq m\leq x+\Delta}A(m,1,...,1)\right|^2\,\mathrm d x\nonumber\\
&=\frac1X\int_X^{2X}\left|P\left(x+\Delta;\theta\right)-P(x;\theta)\right|^2\,\mathrm d x+\frac1X\int_X^{2X}\left|E\left(x+\Delta;\theta\right)-E(x;\theta)\right|^2\,\mathrm d x\nonumber\\
&+O\left(\frac1X\int_X^{2X}\left|P\left(x+\Delta;\theta\right)-P(x;\theta)\right|\cdot\left|E\left(x+\Delta;\theta\right)-E(x;\theta)\right|\mathrm d x\right)
\end{align}
for any $0<\theta\leq 1$. In fact, we will suppose that $0<\theta< 1/(n-1+2n\vartheta)$. For the first term on the right-hand side we see that 
\begin{align}\label{137}
&\frac1X\int_X^{2X}\left|P\left(x+\Delta;\theta\right)-P(x;\theta)\right|^2\,\mathrm d x\nonumber\\
&=\frac1X\int_X^{2X}|P(x+\Delta;\theta)|^2\,\mathrm d x+\frac1X\int_X^{2X}|P(x;\theta)|^2\,\mathrm d x\nonumber\\
&-\frac1X\int_X^{2X}\left[P(x+\Delta;\theta)\overline{P(x;\theta)}+P(x)\overline{P(x+\Delta;\theta)}\right]\,\mathrm d x.
\end{align}
By writing cosines as exponentials we have
\begin{align*}
P(x;\theta)&=\frac{x^{1/2-1/2n}}{2\pi\sqrt n}\sum_{m\leq X^\theta}\frac{A(1,...,1,m)}{m^{1/2+1/2n}}e\left(n\sqrt[n]{mx}+\frac{(n-3)}8\right)\\
&+\frac{x^{1/2-1/2n}}{2\pi\sqrt n}\sum_{m\leq X^\theta}\frac{A(1,...,1,m)}{m^{1/2+1/2n}}e\left(-n\sqrt[n]{mx}-\frac{(n-3)}8\right).
\end{align*}
Arguing just as in the proof of Lemma 10 we see that 
\begin{align*}
\frac1X\int_X^{2X}\left|P(x;\theta)\right|^2\,\mathrm d x\sim\frac12\cdot\frac1{n\pi^2}\cdot\frac{2^{2-1/n}-1}{2-1/n}\cdot X^{1-1/n}\sum_{m=1}^\infty\frac{|A(m,1,...,1)|^2}{m^{1+1/n}}.
\end{align*}
assuming $\theta<1/(n-1)-\varepsilon$. The identical argument shows that
\begin{align*}
\frac1X\int_X^{2X}\left|P(x+\Delta;\theta)\right|^2\,\mathrm d x
\end{align*}
satisfies the same asymptotics under the additional condition $\Delta=o(X)$. 

On the other hand, in order to estimate the last remaining term in (\ref{137}) a short calculation by writing cosines in terms of exponential functions shows that we need to estimate integrals of the form
\begin{align*}
\frac1X\int_X^{2X}(x(x+\Delta))^{1/2-1/2n}e\left(\pm n\left(\sqrt[n]{m(x+\Delta)}\pm\sqrt[n]{\ell x}\right)\pm\mu\cdot\frac{(n-3)}4\right)\,\mathrm d x
\end{align*}
with $\mu\in\{0,1\}$. 

Set $F(x):=\sqrt[n]{m(x+\Delta)}-\sqrt[n]{\ell x}$. Using the easy observation that for $x\neq y$ we have $|\sqrt[n]x-\sqrt[n]y|\gg|x-y|(\max(x,y))^{1/n-1}$, it follows that 
\begin{align*}
|F'(x)|\gg X^{1/n-1}|m-\ell|(\max(m,\ell))^{1/n-1}
\end{align*}
for $m\neq\ell$. Also
\begin{align*}
\sqrt[n]{m(x+\Delta)}-\sqrt[n]{mx}=\sqrt[n]m\int_x^{x+\Delta}t^{1/n-1}\,\mathrm d t\asymp\sqrt[n]m\Delta X^{1/n-1}.
\end{align*}
Therefore, by applying the first derivative test, we have 
\begin{align*}
&\frac1X\int_X^{2X}\frac{(x(x+\Delta))^{1/2-1/2n}}{4\pi^2n}e\left(\pm n\left(\sqrt[n]{m(x+\Delta)}-\sqrt[n]{\ell x}\right)\pm\mu\cdot\frac{(n-3)}4\right)\,\mathrm d x\\
&\ll\left\{  
\begin{array}{l l}
\frac{X^{1-2/n}(\max(m,\ell))^{1-1/n}}{|m-\ell|},\,\,\quad\text{if }m\neq \ell \\
\frac{X^{2-2/n}}{\Delta\sqrt[n]m},\qquad\qquad\qquad\quad\,\,\,\text{if }m=\ell
\end{array}\right.
\end{align*}
Similarly, 
\begin{align*}
&\frac1X\int_X^{2X}\frac{(x(x+\Delta))^{1/2-1/2n}}{4\pi^2n}e\left(\pm n\left(\sqrt[n]{m(x+\Delta)}+\sqrt[n]{\ell x}\right)\pm\mu\cdot\frac{(n-3)}4\right)\,\mathrm d x\\
&\ll\frac{X^{1-2/n}(\max(m,\ell))^{1-1/n}}{|m+\ell|}.
\end{align*}
Hence, the non-diagonal terms in 
\begin{align*}
\frac1X\int_X^{2X}|P(x+\Delta;\theta)P(x;\theta)|\,\mathrm d x
\end{align*} 
contribute
\begin{align*}
&\ll_\varepsilon X^{1-2/n+\theta(1-1/n)+(n-2)\varepsilon/n}\\
&\ll_\varepsilon X^{1-1/n-\varepsilon/n}
\end{align*}
by using  (\ref{final-bound}) and the assumption $\theta<1/(n-1)-\varepsilon$. The diagonal contribution is estimated as
\begin{align*}
&\ll\frac{X^{2-2/n}}{\Delta}\sum_{m\leq X^\theta}\frac{|A(1,...,1,m)|^2}{m^{1+1/n}}\\
&\ll \frac{X^{2-2/n}}{\Delta}\\
&\ll_\varepsilon X^{1-1/n-\varepsilon},
\end{align*}
provided that $\Delta\gg_\varepsilon X^{1-1/n+\varepsilon}$, again by using (\ref{final-bound}). 

The term involving $E(x;\theta)$ is $\ll_\varepsilon X^{1-1/n-\varepsilon}$, which follows from the proof of Proposition 8 (here the generalised Lindel\"of hypothesis is needed) for $0<\theta<1/(n-1+2n\vartheta)$ assuming $\vartheta<1/2-1/n$. Finally, the error term in (\ref{1372}) is $\ll_\varepsilon X^{1-1/n-\varepsilon/2n}$ by the Cauchy-Schwarz inequality. This concludes the proof. \qed

\section{Acknowledgements}

The author was supported by the doctoral program DOMAST of the University
of Helsinki and the Finnish Cultural Foundation. The author thanks Dr.\,Anne-Maria Ernvall-Hyt\"onen for various useful comments and discussions. The author also thanks Dr.\,Anders S\"odergren and Prof.\,Morten Risager for comments on the previous version of the present article. 
 
\end{document}